\documentclass[11pt]{article}
\usepackage[dvips]{graphicx}
\usepackage{ amssymb,amsbsy  }
\usepackage{ifthen}
\usepackage{cite}

\usepackage{amsfonts,amsmath,amsthm}
\usepackage[margin=1.8cm]{geometry}
\usepackage[algo2e,linesnumbered,vlined,ruled]{algorithm2e}
\usepackage{xcolor}
\usepackage{verbatim}
\usepackage{dsfont}
\usepackage{hyperref}
\usepackage{fancyhdr}
\usepackage[page,title]{appendix}
\newtheorem{theorem}{Theorem}[section]
\newtheorem{lemma}[theorem]{Lemma}

\theoremstyle{definition}
\newtheorem{definition}[theorem]{Definition}
\newtheorem{example}[theorem]{Example}

\newtheorem{proposition}[theorem]{Proposition}
\theoremstyle{remark}
\newtheorem{remark}[theorem]{Remark}

\numberwithin{equation}{section}
\usepackage{setspace,url}
\usepackage{graphics,graphicx,epstopdf}
\usepackage{tabularx}
\usepackage{longtable}
\usepackage{booktabs,multirow,array,multicol}
\usepackage{enumitem}
\usepackage{subfig}

\ifpdf
  \DeclareGraphicsExtensions{.eps,.pdf,.png,.jpg}
\else
  \DeclareGraphicsExtensions{.eps}
\fi
\graphicspath{ {./results/} }

\DeclareMathOperator*{\argmin}{arg\,min}
\newcommand{\be}{\begin{equation}}
\newcommand{\ee}{\end{equation}}

\title{Second-order optimality conditions for bilevel
programs\thanks{\baselineskip 9pt This work was supported by National Key R\&D Program of China under project No. 2022YFA1004000.}}
\author{Xiang Liu\thanks{\baselineskip 9pt School of Mathematical Sciences, Dalian University of Technology, Dalian 116024, China. E-mail: liuxiang@mail.dlut.edu.cn}, \
Mengwei Xu\thanks{\baselineskip 9pt Institute of Mathematics, Hebei University of Technology, Tianjin 300401, China.
E-mail: xumengw@hotmail.com. The
research of this author was supported by the National Natural
Science Foundation of China under project No. 12071342 and the Natural Science Foundation of Hebei Province under project No. A2020202030} \   \ and \ Liwei Zhang\thanks{\baselineskip 9pt
School of Mathematical Sciences, Dalian University of Technology, Dalian 116024, China. E-mail: lwzhang@dlut.edu.cn. The
research of this author was supported by the National Natural Science Foundation of China under project  No. 11971089 and partially supported by Dalian High-level Talent Innovation Project under project  No. 2020RD09. }
}

\begin{document}
\maketitle
\begin{abstract}
 Second-order optimality conditions of the
bilevel programming problems are  dependent on the second-order
directional derivatives of the value functions or the solution
mappings of the lower level problems under some regular conditions,
which can not be calculated or evaluated. 
To overcome this difficulty, we propose the notion of the bi-local solution.
 Under the Jacobian uniqueness conditions for the lower level
 problem, we prove that the bi-local solution is a local minimizer of certain one-level minimization problems. Basing on this property, the first-order necessary optimality conditions
 and second-order necessary and sufficient optimality conditions for the
bi-local optimal solution of a given bilevel program are established.
 The second-order optimality conditions proposed here
only involve second-order derivatives of the defining functions of the bilevel problem. 
The second-order sufficient optimality conditions are used to derive the   Q-linear convergence rate of the classical augmented
Lagrangian method.
\end{abstract}

{\bf 2020 Mathematics Subject Classification.} 90C26, 90C30, 90C46 \\

{\bf Keywords:}  Bilevel programs, first-order optimality conditions,
second-order optimality conditions, Jacobian uniqueness conditions,
augmented Lagrangian methods.\\

\newpage

\baselineskip 18pt
\parskip 2pt

\section{Introduction}
Bilevel programming problems originated from the economic game theory by Stackelberg in 1934 \cite{Stackelberg},
which has hierarchical structure that the constraint region  involves the graph of the solution set of the parametric problem, whose parameters are variables of the upper level problem.
Bilevel optimization has a wide range of applications, including economics, finance and logistics, as well as hyper-parameter selection in the fast-growing area of machine learning. For a detailed introduction to bilevel optimization, see \cite{Shimizu 97, Bard 98, Dempe 02, Dempe 15, Dempe 20, Kunapuli 08} and references therein.

In this paper, we consider the following bilevel optimization problem:
\begin{eqnarray*}
(BP)~~~~\min && F(x,y)\\
{\rm s.t.}&& H(x,y)= 0,\\
&& G(x,y)\le 0,\\
&& y\in S(x),
\end{eqnarray*}
where $ S(x)$ denotes the solution set of the lower level problem:
\begin{eqnarray*}
(P_x)~~~~\min && f(x,y)\\
{\rm s.t.} && y\in Y(x).
\end{eqnarray*}
Here $Y(x):=\{y\in \mathbb{R}^m: h(x,y)=0,\enspace g(x,y)\le0\}$
and $F:\mathbb{R}^n \times \mathbb{R}^m \to \mathbb{R} $, $G :\mathbb{R}^n\times \mathbb{R}^m\to \mathbb{R}^q $, $H :\mathbb{R}^n\times \mathbb{R}^m\to \mathbb{R}^p $,
$f:\mathbb{R}^n \times \mathbb{R}^m \to \mathbb{R} $,  $g :\mathbb{R}^n\times \mathbb{R}^m\to \mathbb{R}^s $, $h :\mathbb{R}^n\times \mathbb{R}^m\to \mathbb{R}^r $.
Let
$
  \Phi:=\{(x,y)\in \mathbb{R}^n \times \mathbb{R}^m:H(x,y)=0,\enspace G(x,y)\le0\}
$ and denote the feasible region of (BP) by $\cal{F}$.

The optimality conditions for bilevel programming problem mainly depend on different single-level reformulations.
Under the assumption that the lower level problem has a unique solution, i.e., $S(x)=\{y(x)\}$, the bilevel problem can be reformulated as a single-level problem by substituting the lower level solution into the upper level problem:
\begin{eqnarray*}
(SP)~~~~\min_x && F(x,y(x))\\
{\rm s.t.}&& H(x,y(x))= 0,\\
&& G(x,y(x))\le 0.
\end{eqnarray*}
 Due to the implicit structure of the lower level solution mapping,
Dempe et al. \cite{Dempe 92, Falk 95, Mehlitz 17, Outrata, Shimizu 97, Zaslavski 12} developed the first-order optimality conditions based on the differential properties of the solution mapping under different regularity conditions of the lower level problem.

In practice, people usually adopt the so-called first-order approach. The first-order approach is to replace
the constraint $y\in S(x)$ by the Karush–Kuhn–Tucker (KKT) conditions of the lower level problem and minimize the original variables and  the corresponding multipliers,
see e.g. \cite{Dempe 92,Dempe 13}:
 \begin{eqnarray*}
(FP)~~~~~~\min_{x,y,\mu,\xi} && F(x,y)\\
{\rm s.t.}&& H(x,y)= 0,\\
&& G(x,y)\le 0,\\
&&\nabla_{y} \mathcal{L}(x;y,\mu,\xi)=0,\\
&&h(x,y)=0,\\
&&g(x,y)-\Pi_{\mathbb{R}_{-}^s}(g(x,y)+\xi)=0,
\end{eqnarray*}
where $\mathcal{L}(x;y,\mu,\xi):=f(x,y)+\mu^Th(x,y)+\xi^Tg(x,y)$
is the Lagrangian function of Problem $(P_x)$.
Dempe et al. \cite{Dempe 06} discussed the first-order necessary and sufficient optimality conditions for the bilevel problems with special structures via the contingent cone of the graph of the lower level solution set. It is showed in \cite[Theorem 3.2]{Dempe 12},
 the local solutions of (FP) and (BP) may not be related if the lower level problem is not convex or certain constraint qualifications fail.

Another approach is based on the value function reformulation. Outrata \cite{Outrata} proposed the value function reformulation for a numerical purpose by defining the optimal value function of the lower level problem.
In \cite{Dempe 07, Dempe 13, Dempe 14}, the authors analyzed the subgradients of value functions in variational analysis to derive necessary optimality conditions for bilevel programs.
 Ye and Zhu \cite{Ye 10} derived the first-order necessary optimality conditions for the bilevel programs by combining the first-order approach and the value function approach.
 Bai and Ye \cite{Ye 21} studied the directional necessary optimality conditions under the directional calmness condition.
In \cite{Aboussoror 18}, the authors derived global optimality conditions, which are expressed in terms of max–min conditions with linked constraints for bilevel optimization problems.


Although there are many results about the first-order optimality conditions for bilevel programming problems,
works on standard second-order optimality conditions are very limited.
 Falk and Liu \cite{Falk 95} discussed the second-order optimality conditions for the case where the upper level problem is unconstrained and the lower level problem  satisfies the strong second-order sufficient optimality conditions and the linear independence constraint qualification (LICQ).
 Dempe and Gadhi \cite{Dempe 10} studied second-order optimality conditions by using the first- and second-order approximations of the value function.
Ye and Zhu \cite{Ye 95} used the value function approach to establish the second-order sufficient optimality conditions (SOSC) for bilevel programs where the upper and lower level problems are both unconstrained.
The second-order necessary conditions were also derived in \cite{Ye 95} by assuming the lower level problem has a unique solution.
Mehlitz and Zemkoho \cite{Zemkoho 21} obtained first-order sufficient optimality conditions by estimating the tangent cone to the feasible set of the bilevel program and second-order sufficient conditions
by using second-order directional derivatives of all the involved functions as well as the value function of the lower level problem.

 The existing works for the second-order optimality conditions of bilevel programs were all based on the second-order (directional) derivatives of the solution mappings or the value functions of the lower level problems. However, due to the implicit structures,  it is hardly to calculate or evaluate   the second-order  (directional) derivatives of the solution mappings or the value functions.
 To deal with such difficulty, we consider the constrained bilevel problem (BP) and study the differentiability of the lower level solution mapping  under the  Jacobian uniqueness conditions for the lower level problem (see the Definition \ref{jaco}).
By assuming the Jacobian uniqueness conditions hold, we observe that the implicit function reformulation (SP) and first-order reformulation (FP) of the bilevel program are equivalent in the sense of local solutions, both of which equal to the so-called bi-local solution of the original problem (BP) (see Definition \ref{def 2.1}).
The bi-local solution we proposed in this paper uses the concept of local solution of the lower level problem, which reduces to the classical local solution under the  convexity of the lower level problem. It is shown that the first- and second-order optimality conditions of (SP) and (FP) are equivalent, respectively.
 Furthermore, the resulting second-order optimality conditions of (FP) do not involve any implicit items.
 As a result, the second-order sufficient optimality conditions can be used to obtain bi-local solutions of the bilevel problems and guarantee the convergence rates of many numerical algorithms, such as the classical augmented Lagrangian algorithm.

  We summarize the main contributions as follows:
  \begin{itemize}
    \item Under the Jacobian uniqueness conditions, the local solutions of (FP) coincide with the bi-local solutions of (BP), which indicates that it is reasonable to deal with (FP) even if the lower level problem is not convex.
    \item The second-order optimality conditions for the bi-local solutions, which are derived through the problem (FP), only involve the second-order derivatives of the defining functions of (BP), which are available beforehand.
  \end{itemize}

The remaining parts of this paper are organized as follows. In Section \ref{section 2}, we explore a new definition of the concept of local solution for the bilevel problem and study the differentiability of the solution mapping of the lower level problem in Section \ref{section 3}. Section \ref{section 4} develops two types of necessary and sufficient optimality conditions and derives primal-dual Q-linear convergence rate of the classical augmented Lagrangian algorithm. Section 5 concludes the paper.

We end this section by describing  the notation in this paper. We denote by
$\|\cdot\|$ the $l_2$-norm of a vector $x$ and by $I_{n}$ the $n\times n$ identity matrix, respectively.
 For a vector $x$, we denote $\boldsymbol{B}_{\delta}(x)=\{x':\parallel x'-x\parallel\le \delta\}$. For a convex set $D\subseteq \mathbb{R}^n$, we denote by $\Pi_D(w)$ the projection of $w$ onto $D$ and define $d(z,D)=\inf_{y\in D}\|z-y\|$ as the distance from $z$ to $D$. For vectors $u,v\in \mathbb{R}^n$, we use $u\bot v$ to denote $u^Tv=0$. For a cone $K\subset \mathbb{R}^q$, its polar cone is $K^{\circ}=\{v\in \mathbb{R}^q|u^Tv\le 0 \enspace\text{for all}\enspace u\in K\}$. We use the notation $\psi (t)=o(t)$ for any function $\psi: \mathbb{R}_+\rightarrow \mathbb{R}^q $ such that $\lim_{t\to 0_{+}}t^{-1}\psi(t)=0$.
   For a set-valued mapping $S:\mathbb{R}^n \rightrightarrows\mathbb{R}^m $,
  we denote the graph of $S$ by
${\rm gph} S(x)=\{(x,y)\in \mathbb{R}^n \times \mathbb{R}^m: y\in S(x)\}$.
 For a function $f :\mathbb{R}^n \times \mathbb{R}^m \to \mathbb{R} $, a mapping $g :\mathbb{R}^n \times \mathbb{R}^m \to \mathbb{R}^q $ and a mapping $y:\mathbb{R}^n\to \mathbb{R}^m$,
 for a fixed point $(x,y)$ with $y=y(x)$, we use $\nabla_x f(x,y(x))$ to denote $\nabla_x f(x,y)|_{y=y(x)}$; $\nabla_y f(x,y(x))$ to denote $\nabla_y f(x,y)|_{y=y(x)}$; $\nabla_{xx}^2 f(x,y(x))$ to denote $\nabla_{xx}^2 f(x,y)|_{y=y(x)}$;  $\nabla_{xy}^2 f(x,y(x))$ to denote $\nabla_{xy}^2 f(x,y)|_{y=y(x)}$;
 $ \mathcal{J}_x g(x,y(x))$ to denote $\mathcal{J}_x g(x,y)|_{y=y(x)}$ and $ \mathcal{J}_y g(x,y(x))$ to denote $\mathcal{J}_y g(x,y)|_{y=y(x)}$.

\section{Solution concepts}\label{section 2}
In nonlinear programming problems, a local optimal solution is defined in a certain sufficiently small neighborhood such that there is no better feasible point with lower  objective value.
While in bilevel programming we have to use an auxiliary problem to illustrate the definition.
Several different notions of local optimal solutions for bilevel programs were discussed in \cite[Section 5.1]{Dempe 02}.
Among them, the following concept is the most popular (see e.g. \cite{Ye 10}):
a point $(x^*,y^*)\in \Phi$ is called a local optimal solution for the problem (BP) if $y^*\in S(x^*)$ and
there exists $\delta>0$ such that
$$F(x^*,y^*)\leq F(x,y),\quad \forall\ (x,y)\in {\cal{F}}\cap \boldsymbol{B}_{\delta}(x^*,y^*).$$
It is called a global optimal solution if $\delta =\infty$ can be selected.

It is well-known that the bilevel program is a difficult problem even when all defining functions are linear. The main difficulty is that the feasible region is implicit by the solution set of a parametric problem, which may be difficult to solve when the problem is nonconvex with respect to the lower level variable.
Another drawback of the concept of local optimal solution is that any isolated feasible point becomes a local solution automatically since
the intersection of the feasible region and any sufficiently small neighborhood is empty.

The following example shows that a global maximum point
of a bilevel problem becomes  a local solution if it is an isolated feasible point.

\begin{example}\label{Counterexample}
Consider the following bilevel problem:
\begin{eqnarray*}
\min && F(x,y):=y\\
{\rm s.t.} && x\in [-1,1],\\
&& y\in S(x):=\arg\min\limits_{y}\ f(x,y):=x^2+y^2\\
&&~~~~~~~~~~~~~~~~~~~~{\rm s.t.}\quad
g(x,y):=(x^2-y-1)(x^2+y^2-1)\leq 0.
\end{eqnarray*}
The feasible region of the lower level problem lies in the light part in Figure 1.
The feasible set of the bilevel problem is $\{(x,y):x\in [-1,1], y=x^2-1\}\cup \{(0,1)\}$, see the red curve and an isolated point in Figure 1.
The global optimal solution is $(x^*,y^*)=(0,-1)$ with the optimal value $-1$.
It is easy to see that $(0,1)$ is a local optimal solution since it is an isolated feasible point. In fact,  $(0,1)$ is a global maximum point with the objective value $F(0,1)=1$.
\begin{figure}
  \centering
    \includegraphics[width=0.6\textwidth]{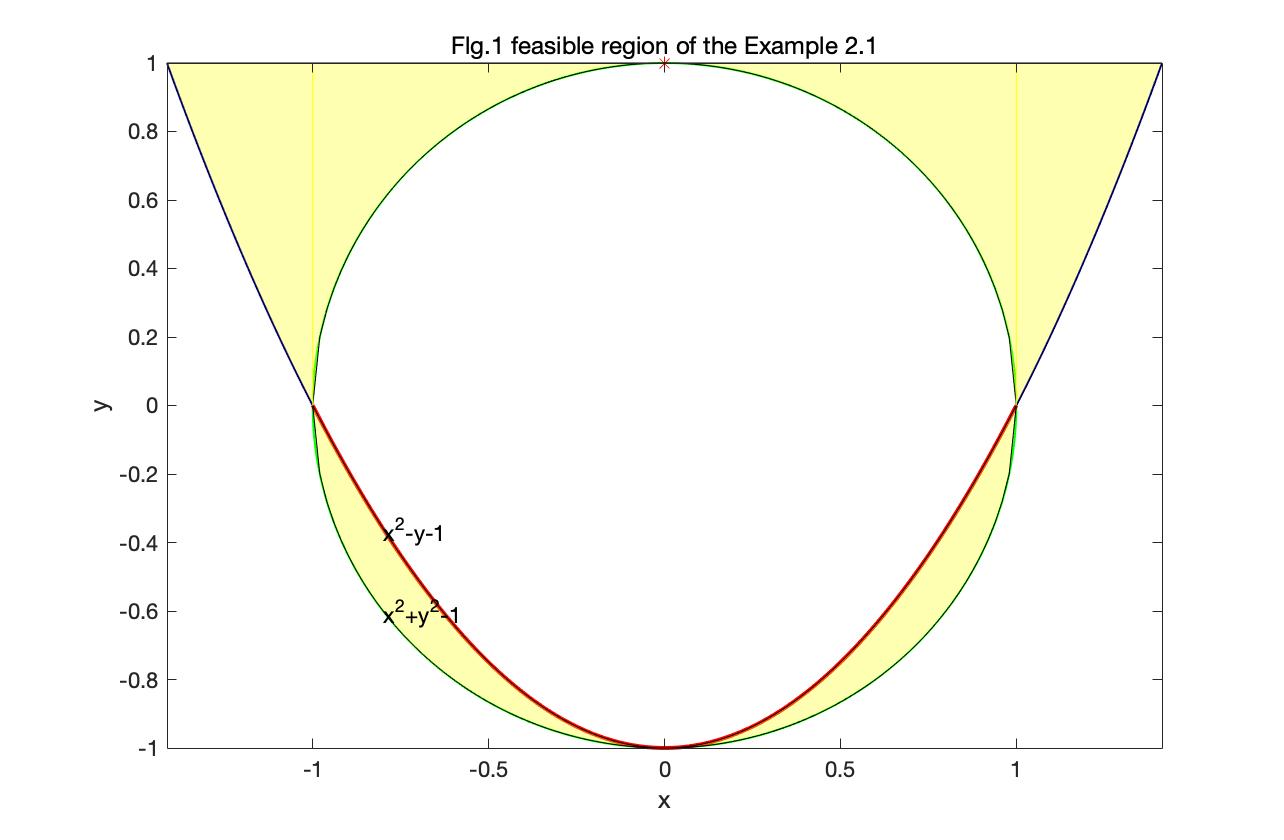}
   \label{fig1}
\end{figure}

\end{example}

The minimax optimization problem is essentially a bilevel programming problem.
Jin, Netrapalli and Jordan \cite{Jin} carefully considered a very basic question: what is a proper definition of local optima of a nonconvex-nonconcave minimax optimization problem.
They proposed a definition called local minimax point for unconstrained minimax optimization problem by using the concept of local solutions of the inner problem, which is applied to the constrained minimax problem by
Dai and Zhang \cite{Daiyh and Zhanglw}.
We now extend this definition to the bilevel problem.

\begin{definition}\label{def 2.1}
A point $\left(x^*, y^*\right) \in \mathbb{R}^n \times \mathbb{R}^m$ is said to be a bi-local optimal solution of (BP) if $(x^*, y^*)\in \Phi$,
$y^*\in Y(x^*)$ and there exist $\delta_0>0$ and a function $\eta:\left(0, \delta_0\right] \rightarrow \mathbb{R}_{+}$ satisfying $\eta(\delta) \rightarrow 0$ as $\delta \rightarrow 0$, such that for any $\delta \in\left(0, \delta_0\right]$,
$$f(x^*,y^*)\le f(x^*,z),\quad \forall z\in Y(x^*)\cap \boldsymbol{B}_{\eta(\delta)}(y^*), $$
and any $(x,y)\in \Phi \cap[\boldsymbol{B}_{\delta}(x^*)\times\boldsymbol{B}_{\eta(\delta)}(y^*)]$ satisfying
   $y\in {\rm argmin}_z\{ f(x,z):z\in Y(x)\cap\boldsymbol{B}_{\eta(\delta)}(y^*)\}$,
it follows that
$$ F(x^*,y^*)\le F(x,y).$$
\end{definition}
Obviously, $(0,1)$ is not a bi-local solution of the Example \ref{Counterexample} and $(0,-1)$ is the unique bi-local solution.

\begin{remark}\label{biloc}
Note that if $\left(x^*, y^*\right)\in \mathbb{R}^n \times \mathbb{R}^m$ is a bi-local solution of (BP) and $y^*\in S(x^*)$, it follows that $\left(x^*, y^*\right)$ is a local solution of (BP).

Indeed, if $(x^*,y^*)$ is isolated in ${\cal{F}}$, it is a local solution of (BP) automatically.
Otherwise we assume to the contrary that there exists an infinite sequence $(x^k,y^k)\in  \Phi$ converging to $\left(x^*, y^*\right)$ as $k\to \infty$ such that
$y^k\in S(x^k)$ and
\begin{equation}\label{con0}
F(x^k,y^k)<F(x^*,y^*).
\end{equation}
For any $\delta_0>0$ and  $\eta:\left(0, \delta_0\right] \rightarrow \mathbb{R}_{+}$ satisfying $\eta(\delta) \rightarrow 0$ as $\delta \rightarrow 0$, the condition $y^k\in S(x^k)$ implies that for any $\delta \in\left(0, \delta_0\right]$, there exists $k_{\delta}$ sufficiently large such that
$(x^{k_{\delta}},y^{k_{\delta}})\in \Phi\cap [\boldsymbol{B}_{\delta}(x^*)\times\boldsymbol{B}_{\eta(\delta)}(y^*)]$
and
$$ f(x^{k_{\delta}},y^{k_{\delta}})\le f(x^{k_{\delta}},z), \quad \forall z\in Y(x^{k_{\delta}})\cap \boldsymbol{B}_{\eta(\delta)}(y^*).$$
Then the condition (\ref{con0}) contradicts with the fact that $\left(x^*, y^*\right) $ is a bi-local solution of (BP).

\end{remark}

The following conclusion shows that the bi-local solution reduces to the so-called local solution when the lower level problem is convex in $y$.

\begin{theorem}\label{loc}
Let $\left(x^*, y^*\right) \in \mathbb{R}^n \times \mathbb{R}^m$ be a bi-local solution of (BP). Assume that the lower level problem $(P_{x})$ is convex at $x^*$ with respect to the variable $y$.
Then $\left(x^*, y^*\right)$ is a local solution of (BP).
Conversely, let $\left(x^*, y^*\right) \in \mathbb{R}^n \times \mathbb{R}^m$ be a global/local solution of (BP)
around which $f,g,h$ are  continuous.
If the solution mapping $S(x)$ of the lower level problem $(P_{x})$ is continuous at $x^*$,
 then $\left(x^*, y^*\right)$ is a bi-local solution of (BP).
\end{theorem}
\begin{proof} Let $\left(x^*, y^*\right) \in \mathbb{R}^n \times \mathbb{R}^m$ be a bi-local solution of (BP). From the Remark \ref{biloc}, we only need to show that under the convexity of the lower level problem, $y^*$ must be a global solution of the problem $(P_{x^*})$, i.e., $y^*\in S(x^*)$.  Otherwise we select $\widetilde{y}\in Y(x^*)$ such that $f(x^*,\widetilde{y})<f(x^*,y^*)$.
For $\alpha_t\in [0,1]$, let $y^t:=\alpha_t \widetilde{y} +(1-\alpha_t) y^*\in Y(x^*)$ from the convexity.
Moreover,
\begin{align*}
      f(x^*,y^t)\leq \alpha_t f(x^*,\widetilde{y})+(1-\alpha_t) f(x^*,y^*)<  f(x^*,y^*).
\end{align*}
By setting $\alpha_t\to 0$, $y^t\in \boldsymbol{B}_{\eta(\delta)}(y^*)$, which is a contradiction with the definition of the bi-local solution. Thus $y^*\in S(x^*)$ and hence $\left(x^*, y^*\right)$ is a local solution of (BP).

Conversely, since $S(x)$ is continuous at $x^*$ and $y^*\in S(x^*)$, we have $d(y^*,S(x))\to 0$, as $x\to x^*$.
Then we denote that
$\eta (\delta):=\sup\limits_{\|x-x^*\|\le \delta}\inf\limits_{y\in S(x)}\|y -y^*\|$
, which satisfies $\eta(\delta) \rightarrow 0$ as $\delta \rightarrow 0$. Since $(x^*, y^*)$ is a global/local solution of (BP), there exist $\delta_0>0$, for any $\delta \in(0, \delta_0]$ and any $(x,y)\in \Phi\cap {\rm gph}S(x) \cap[\boldsymbol{B}_{\delta}(x^*)\times\boldsymbol{B}_{\eta(\delta)}(y^*)]$,
it follows that
$$ F(x^*,y^*)\le F(x,y).$$
This indicates $(x^*,y^*)$ is a bi-local solution of (BP) by $\argmin_z\{f(x,z):z\in Y(x)\cap \boldsymbol{B}_{\eta(\delta)}(y^*)\}\subseteq S(x)$.
We complete the proof.
\end{proof}

\section{Differentiability  of the solution mapping of the lower level problem}\label{section 3}
In this section, for a given point $x^*$, we study the continuity property of the solution mapping of the lower level problem under the Jacobian uniqueness conditions for the problem $(P_{x^*})$ at a certain point $(y^*, \mu^*,\xi^*)$.
Note that $(\mu^*,\xi^*)\in \mathbb{R}^r\times\mathbb{R}^s$ is the corresponding Lagrangian multiplier such that
the KKT conditions hold at $y^*$ for Problem $(P_{x^*})$, i.e.,
 \begin{eqnarray*}
 && \nabla_y \mathcal{L}\left(x^*; y^*, \mu^*,\xi^*\right)=0,\\
&& h\left(x^*, y^*\right)=0,\\
&& 0\le \xi^*\bot g(x^*,y^*)\leq 0.
\end{eqnarray*}
Recall that $\mathcal{L}(x;y,\mu,\xi):=f(x,y)+\mu^Th(x,y)+\xi^Tg(x,y)$.

In order to state the optimality conditions for (BP), we require the following conditions.
\begin{definition}\label{jaco} Let $\left( \mu^*,\xi^*\right) \in \mathbb{R}^r \times \mathbb{R}^s$ be a point. We say that the Jacobian uniqueness conditions of Problem $\left(\mathrm{P}_{x^*}\right)$ are satisfied at $\left(y^*, \mu^*,\xi^*\right)$ if
\begin{itemize}
  \item [(a)] The point $(y^*, \mu^*,\xi^*)$ is a KKT pair of Problem $\left(\mathrm{P}_{x^*}\right)$.

  \item [(b)] The LICQ holds at $y^*$; namely, the set of vectors
$$
\left\{\nabla_y h_1(x^*, y^*), \cdots, \nabla_y h_r(x^*, y^*)\right\}\cup\{\nabla_y g_i(x^*, y^*):i\in I_{x^*}(y^*)\}
$$
is linearly independent, where $I_{x^*}(y^*)=\{i=1,\cdots,s:g_i(x^*,y^*)=0\}$.
\item [(c)] The strict complementarity condition holds at $y^*$ for $\xi^*$; namely,
$$ \xi^*_i-g_i(x^*,y^*)>0, \quad i=1,\cdots, s. $$
\item [(d)] The SOSC holds at $(y^*, \mu^*,\xi^*)$; namely,
$$
\left\langle\nabla_{y y}^2 \mathcal{L}\left(x^*; y^*, \mu^*,\xi^*\right) d_y, d_y\right\rangle>0, \quad \forall d_y \in \mathcal{C}_{x^*}(y^*)\backslash\{0\}, $$
where $\mathcal{C}_{x^*}(y^*)$ is the critical cone of Problem $(P_{x^*})$ at $y^*$,
\begin{align*}
  \mathcal{C}_{x^*}(y^*)=\{d_y\in \mathbb{R}^m: &\mathcal{J}_y h(x^*, y^*)d_y=0;\nabla_y g_i(x^*,y^*)^Td_y\le 0,i\in I_{x^*}(y^*);\\
  &\nabla_y f(x^*,y^*)^Td_y\le 0 \}  .
\end{align*}
\end{itemize}
\end{definition}

Recall the following results from \cite[Theorem 2.1]{Robinson 74}.
\begin{lemma}\label{lemma-IFT}
 Let $(x^*, y^*) \in \mathbb{R}^n \times \mathbb{R}^m$ be a point around which $f,g,h$ are  continuously differentiable and
twice continuously differentiable with respect to the variable $y$. 
 Assume that the Jacobian uniqueness conditions of Problem $(\mathrm{P}_{x^*})$ are satisfied at $(y^*, \mu^*,\xi^*)$. Then there exist $\delta_1>0$, $\varepsilon_1>0$, and a
 continuously differentiable mapping $(y, \mu,\xi): \boldsymbol{B}_{\delta_1}(x^*) \rightarrow \boldsymbol{B}_{\varepsilon_1}(y^*) \times\boldsymbol{B}_{\varepsilon_1}(\mu^*)\times\boldsymbol{B}_{\varepsilon_1}(\xi^*)$ such that the Jacobian uniqueness conditions of Problem $(\mathrm{P}_x)$ are satisfied at $(y(x), \mu(x),\xi(x))$ when $x \in \boldsymbol{B}_{\delta_1}(x^*)$.
\end{lemma}
\begin{remark}\label{second IFT}
From the second-order implicit function theorem \cite[page 364]{Lang 93}, we note that if $f,g,h$ are twice continuously differentiable and thrice continuously differentiable with respect to the variable $y$ around $(x^*, y^*)$, then the mapping $(y(x), \mu(x),\xi(x))$ is twice continuously differentiable for $x \in \boldsymbol{B}_{\delta_1}(x^*)$.
\end{remark}

The following proposition reveals the differentiation of the solution mapping and the corresponding multiplier mappings of the lower level problem.
\begin{proposition}\label{prop3.1}
If the assumptions of Lemma $\ref{lemma-IFT}$ are satisfied, $(y(x),\mu(x),\xi(x))$ is defined in Lemma \ref{lemma-IFT}, then
  \begin{equation}\label{nabla-y}
  \left[
  \begin{array}{c}
    \mathcal{J}y(x) \\
     \mathcal{J}\mu(x)\\
     \mathcal{J}\xi(x)
  \end{array}
  \right]
  =-K(x)^{-1}\left[
		\begin{array}{c}
			\nabla_{yx}^2 \mathcal{L}(x;y(x),\mu(x),\xi(x))\\
            \mathcal{J}_x h(x,y(x))\\
			(I-W)\mathcal{J}_x g(x,y(x))
		\end{array}\right],
\end{equation}
where
\begin{equation}\label{equ-K}
  K(x):=\left[
		\begin{array}{ccc}
			\nabla_{yy}^2 \mathcal{L}(x;y(x),\mu(x),\xi(x))&\mathcal{J}_y h(x,y(x))^T&\mathcal{J}_y g(x,y(x))^T\\
             \mathcal{J}_y h(x,y(x))&0&0\\
			(I-W)\mathcal{J}_y g(x,y(x))&0&-W
		\end{array}\right]
\end{equation}
 and
 $$  W= \mathcal{J}\Pi_{\mathbb{R}_{-}^s}(g(x,y(x))+\xi(x)).$$
\end{proposition}
\begin{proof}
Define the index sets:
\begin{eqnarray*}
&&\alpha=\left\{i: g_i(x,y(x))=0, \xi_i(x)>0\right\},\\
&& \beta=\left\{i: g_i(x,y(x))=0, \xi_i(x)=0\right\},\\
&& \gamma=\left\{i: g_i(x,y(x))<0, \xi_i(x)=0\right\}.
\end{eqnarray*}
Under the strict complementarity condition holds at $y(x)$ for $\xi(x)$, we have that $\beta=\emptyset$ and
 $W$ can be expressed as
$$
W=\operatorname{Diag}\left(w_1, \cdots, w_{s}\right),
$$
where $w_i=0$, $ i\in \alpha$ and $ w_i=1$, $ i\in \gamma $.
Then $K(x)$ can be expressed as
\begin{align*}
K(x)&=\left[\begin{array}{cccc}
\nabla_{y y}^2 \mathcal{L}(x; y(x), \mu(x), \xi(x)) & \mathcal{J}_y h(x,y(x)) ^T & \mathcal{J}_y g_\alpha(x,y(x)) ^T &  \mathcal{J}_y g_\gamma(x,y(x)) ^T \\
\mathcal{J}_y h(x,y(x))  & 0 &  0 & 0 \\
\mathcal{J}_y g_\alpha(x,y(x))  & 0 &0 & 0 \\
0 & 0 & 0   & -I_{|\gamma|}
\end{array}\right].
\end{align*}

We first prove that when $(y(x), \mu(x),\xi(x))$ satisfies the Jacobian uniqueness conditions for Problem $(\mathrm{P}_x)$, $K(x)$ is nonsingular.
For $\eta_1 \in \mathbb{R}^m, \eta_2 \in \mathbb{R}^{r}, \eta_3 \in \mathbb{R}^{|\alpha|}, \eta_4 \in \mathbb{R}^{|\gamma|}$, consider
$$
K(x)\left[\begin{array}{l}
\eta_1 \\
\eta_2 \\
\eta_3\\
\eta_4
\end{array}\right]=0,
$$
or equivalently
\begin{gather}
   \nabla_{y y}^2 \mathcal{L}(x; y(x), \mu(x), \xi(x)) \eta_1+\mathcal{J}_y h(x,y(x))^T \eta_2
  +\mathcal{J}_y g_\alpha(x,y(x)) ^T \eta_3+\mathcal{J}_y g_\gamma(x,y(x)) ^T\eta_4=0,\label{3-11-1}\\
 \mathcal{J}_y h(x,y(x)) \eta_1=0,\quad
\mathcal{J}_y g_\alpha(x,y(x))\eta_1=0,\quad \eta_4 =0.\label{3-11}
\end{gather}
From the first two equations in (\ref{3-11}) and the fact that
$$\nabla_y f(x,y(x))+ \mathcal{J}_yh(x,y(x))^T\mu(x)+ \mathcal{J}_yg(x,y(x))^T\xi(x)=0,$$
we derive that
\begin{align*}
  \eta_1\in \mathcal{C}_x(y(x))=\{d_y\in \mathbb{R}^m:&\mathcal{J}_y h(x, y(x))d_y=0;\mathcal{J}_y g_\alpha(x,y(x))d_y= 0 ;\nabla_y f(x,y(x))^Td_y= 0 \},
\end{align*}
 i.e., the critical cone of Problem $(P_x)$ at $y(x)$.
Multiplying  $\eta_1^T$ to both sides of (\ref{3-11-1}), we obtain
\begin{equation*}
  \eta_1^T\nabla_{y y}^2 \mathcal{L}(x; y(x), \mu(x), \xi(x)) \eta_1+\eta_1^T\mathcal{J}_y h(x,y(x))^T \eta_2
  +\eta_1^T \mathcal{J}_y g_\alpha(x,y(x)) ^T \eta_3
  =0,
\end{equation*}
by  (\ref{3-11}), we have
\begin{equation*}
  \eta_1^T\nabla_{y y}^2 \mathcal{L}(x; y(x), \mu(x), \xi(x)) \eta_1 =0,
\end{equation*}
for $\eta_1\in \mathcal{C}_x(y(x))$. Therefore we obtain $\eta_1=0$ from the SOSC at $(y(x),\mu(x),\xi(x))$ for Problem $(P_x)$.  Substituting $\eta_1=0$ and $\eta_4=0$ to (\ref{3-11-1}) and from the LICQ, we get $\eta_2=0$, $\eta_3=0$. Therefore, matrix $K(x)$ is nonsingular.

It follows from Lemma \ref{lemma-IFT} that the KKT condition holds $(y(x),\mu(x),\xi(x))$ for $(P_x)$, i.e.,
  \begin{equation}\label{GE-1}
	F^{KKT}(x,y(x),\mu(x),\xi(x)):=
	\left(
	\begin{array}{c}
		\nabla_y \mathcal{L}(x;y(x),\mu(x),\xi(x))\\
        h(x,y(x))\\
		g(x,y(x))-\Pi_{\mathbb{R}_{-}^s}(g(x,y(x))+\xi(x))
	\end{array}
	\right)=0,
\end{equation}
and $ \Pi_{\mathbb{R}_{-}^s}(g(x,y(x))+\xi(x)) $ is differentiable since $g(x,y(x))+\xi(x) \neq 0$.

It is easy to calculate that
\begin{eqnarray*}
   \mathcal{J}_{(y,\mu,\xi)}F^{KKT}(x,y(x),\mu(x),\xi(x))=K(x)
\end{eqnarray*}
and
\begin{equation*}
  \mathcal{J}_xF^{KKT}(x,y(x),\mu(x),\xi(x))=
		\left[
		\begin{array}{c}
			\nabla_{yx}^2 \mathcal{L}(x;y(x),\mu(x))\\
           \mathcal{J}_x h(x,y(x))\\
			(I-W)\mathcal{J}_x g(x,y(x))
		\end{array}\right].
\end{equation*}
Differentiating both sides of (\ref{GE-1}) with respect to the variable $x$ yields that
\begin{equation}
  \mathcal{J}_xF^{KKT}(x,y(x),\mu(x),\xi(x))+\mathcal{J}_{(y,\mu,\xi)}F^{KKT}(x,y(x),\mu(x),\xi(x))
		\left[
  \begin{array}{c}
    \mathcal{J}y(x) \\
    \mathcal{J}\mu(x) \\
     \mathcal{J}\xi(x)
  \end{array}
  \right]=0,\notag
\end{equation}
then we obtain the derivatives (\ref{nabla-y}).
\end{proof}

\section{Optimality conditions and  applications to the augmented Lagrangian method}\label{section 4}
In this section, we establish the relationships between the problems (SP) and (FP) as discussed in Section \ref{S 4.1}. Subsequently, we analyze the first-order optimality conditions and the second-order optimality conditions for the bilevel problem (BP) in Section \ref{S 4.2}. Finally, we utilize the second-order sufficient optimality conditions to ensure the convergence rate of the augmented Lagrangian method for the bilevel problem (BP) in Section \ref{S 4.3}.

\subsection{The equivalence of (SP) and (FP)}\label{S 4.1}
In this subsection, we demonstrate the local equivalence of problem (SP) and problem (FP). Additionally, we establish that under the Jacobian uniqueness conditions, the local solutions of (SP) and (FP) serve as bi-local solutions for (BP).

\begin{theorem}\label{prop-equal BP CP}
Assume that the lower level problem $(P_{x^*})$ satisfies the Jacobian uniqueness conditions at $(y^*,\mu^*,\xi^*)$. Then the following statements are equivalent:
\begin{itemize}
  \item[(i)] $(x^*,y^*)$ is a bi-local solution of (BP);
  \item [(ii)] $x^*$ is a local solution of (SP);
  \item[(iii)] $(x^*,y^*,\mu^*,\xi^*)$ is a local solution of (FP).
\end{itemize}
\end{theorem}

\begin{proof}
  It follows from Lemma $\ref{lemma-IFT}$ that $y(x)\in \boldsymbol{B}_{\epsilon_1}(y^*)$ is the unique local minimizer of Problem $(P_x)$ for $x\in \boldsymbol{B}_{\delta_1}(x^*)$. Then $(x^*,y^*)$ is a bi-local solution of (BP) if and only if for any $(x,y(x))\in \Phi \cap[\boldsymbol{B}_{\delta_1}(x^*)\times\boldsymbol{B}_{\epsilon_1}(y^*)]$,
$ F(x^*,y^*)\le F(x,y(x))$, which means that $x$ is a local solution of (SP). The converse is also ture. Thus the relation $(i)\Leftrightarrow (ii)$ holds.

It also follows from Lemma $\ref{lemma-IFT}$ that $(y(x),\mu(x),\xi(x))\in \boldsymbol{B}_{\varepsilon_1}(y^*)\times\boldsymbol{B}_{\varepsilon_1}(\mu^*)
\times\boldsymbol{B}_{\varepsilon_1}(\xi^*)$ is the unique KKT point of Problem $(P_x)$
 for $x\in \boldsymbol{B}_{\delta_1}(x^*)$.
Thus the feasible set of (FP) locally turns out to be $\{(x,y(x),\mu(x),\xi(x)):(x,y(x))\in \Phi\}$ .
Then $(x^*,y^*,\mu^*,\xi^*)$ is a local solution of (FP) if and only if for any $(x,y(x))\in \Phi \cap[\boldsymbol{B}_{\delta_1}(x^*)\times\boldsymbol{B}_{\epsilon_1}(y^*)]  $
such that
$ F(x^*,y^*)\le F(x,y(x))$. Thus the relation $(iii)\Leftrightarrow (ii)$ holds.
We complete the proof.
\end{proof}
\subsection{Optimality conditions}\label{S 4.2}
Constraint qualifications are imposed on the constraint region of a mathematical program to guarantee the KKT conditions hold at local optimal solutions.
Consider the following nonlinear problem:
$$(P)~~~~~~\min ~f(x)~~~{\rm s.t.}~h(x)=0,~g(x)\le0,$$
where $f:\mathbb{R}^n  \to \mathbb{R} $, $h:\mathbb{R}^n \to \mathbb{R}^{n_1} $ and $g:\mathbb{R}^{n} \to \mathbb{R}^{n_2} $. Denote the feasible region of the problem $(P)$ by $\Omega $.
Let $x^*\in \Omega$ be a point around which $f,g,h$ are continuously differentiable.
The point $x^*$ is said to satisfy the
  Mangasarian-Fromovitz constraint qualification (MFCQ)  for the problem $(P)$ if
\begin{itemize}
  \item [(a)]The set of vectors $\{\nabla h_j(x^*), j=1, \cdots, n_1\}$ is linearly independent;
  \item [(b)]There exists a vector $d \in \mathbb{R}^n$ such that
$$
\nabla h_j(x^*)^T d=0, \quad j=1, \cdots, n_1, \quad \nabla g_i(x^*)^T d<0, \quad i \in I(x^*),
$$
where $I(x^*)=\left\{i=1, \cdots, n_2: g_i(x^*)=0\right\}$.
\end{itemize}

We first show that under the Jacobian uniqueness conditions, the MFCQ for $(SP)$ is the same as it for $(FP)$.
\begin{proposition}\label{prop 4-1}
Let $\left(x^*, y^*\right)  \in \Phi$ be a point around which $f,g,h$ are  continuously differentiable and
twice continuously differentiable with respect to the variable $y$ and $F,G,H$ are continuously differentiable. Suppose that the lower level problem $(P_{x^*})$ satisfies the Jacobian uniqueness conditions at $(y^*,\mu^*,\xi^*)$.
Then the MFCQ holds at $x^*$ for (SP) if and only if the MFCQ holds at $(x^*,y^*,\mu^*,\xi^*)$ for (FP).
	\end{proposition}
\begin{proof} It is easy to see that $(x^*,y^*,\mu^*,\xi^*)$ is a feasible point of (FP).  The MFCQ holds at $(x^*,y^*,\mu^*,\xi^*)$ for the problem (FP), if the following conditions hold:
\begin{itemize}
  \item[(a)] The matrix
  \begin{equation}\label{A}
             A:=\left[\begin{array}{cccc}
                     \mathcal{J}_xH(x^*,y^*) & \mathcal{J}_yH(x^*,y^*) & 0 & 0 \\
                    \nabla_{yx}^2 \mathcal{L}(x^*;y^*,\mu^*,\xi^*)
                    &\nabla_{yy}^2 \mathcal{L}(x^*;y^*,\mu^*,\xi^*)
                    &\mathcal{J}_y h(x^*,y^*)^T &\mathcal{J}_y g(x^*,y^*)^T \\
                     \mathcal{J}_x h(x^*,y^*) & \mathcal{J}_y h(x^*,y^*)&0&0 \\
                     (I-W^*)\mathcal{J}_x g(x^*,y^*) & (I-W^*)\mathcal{J}_y g(x^*,y^*)&0& -W^*
                   \end{array}
  \right]
              \end{equation}
  has full row rank, where $ W^*= \mathcal{J}\Pi_{\mathbb{R}_{-}^s}(g(x^*,y^*)+\xi^*).$
  \item[(b)]There is $d= (d_x,d_y,d_{\mu},d_{\xi})\in ker A $ such that 
      \begin{eqnarray}
    &&  \nabla_xG_i(x^*,y^*)^Td_x+  \nabla_yG_i(x^*,y^*)^T d_y  <0,\quad i\in I_G(x^*,y^*), \label{mfg}
  \end{eqnarray}
  where $I_G(x^*,y^*):=\{i=1,\cdots,q:G_i(x^*,y^*)=0\}$.
\end{itemize}

 We prove the equivalence by the following two steps.

(i) Let
$$B:=\left[\begin{array}{cccc}
                I_n & 0 & 0 & 0 \\
                 \mathcal{J}y (x^*) & I_m & 0 & 0 \\
                 \mathcal{J}\mu(x^*) & 0 & I_r & 0 \\
                  \mathcal{J}\xi(x^*) &  0 & 0 & I_s
              \end{array}\right].$$
Since the matrix $B$ is nonsingular, then $rank(A)=rank(AB)$.
                By (\ref{nabla-y}), we have that
 \begin{align*}
   &AB=\left[\begin{array}{cc}
                \mathcal{J}_xH(x^*,y^*)+\mathcal{J}_yH(x^*,y^*)\mathcal{J}y(x^*) & V  \\
                 0 &K(x^*)
              \end{array}\right],
  \end{align*}
  where $V:=[\mathcal{J}_yH(x^*,y^*), 0, 0]$ and $K(x)$ is defined by (\ref{equ-K}).

Since $K(x^*)$  is nonsingular from the proof of Proposition \ref{prop3.1},
  it follows that $A$ has full row rank if and only if $\mathcal{J}_xH(x^*,y^*)+\mathcal{J}_yH(x^*,y^*)\mathcal{J}y(x^*) $ has full row rank.

(ii) Suppose there exists $d= (d_x,d_y,d_{\mu},d_{\xi})\in ker A$, i.e.,
  \begin{eqnarray}
        &\mathcal{J}_xH(x^*,y^*)d_x+\mathcal{J}_yH(x^*,y^*)d_y=0,\label{4-2-1}\\
    & \left[ \begin{array}{c}
           \nabla_{yx}^2 \mathcal{L}(x^*;y^*,\mu^*,\xi^*)
                     \\
          \mathcal{J}_x h(x^*,y^*)  \\
       (I-W^*)\mathcal{J}_y g(x^*,y^*)
         \end{array} \right]d_x +K(x^*)\left[\begin{array}{c}
       d_y\\
       d_{\mu}\\
       d_{\xi}
              \end{array}
         \right]=0.\label{4-2}
  \end{eqnarray}
From Proposition \ref{prop3.1}, the condition $ (\ref{4-2})$ becomes
 \begin{eqnarray}
&&\left[\begin{array}{c}
       d_y\\
       d_{\mu}\\
       d_{\xi}
              \end{array}
         \right]=-K(x^*)^{-1}\left[ \begin{array}{c}
           \nabla_{yx}^2 \mathcal{L}(x^*;y^*,\mu^*,\xi^*)
                     \\
          \mathcal{J}_x h(x^*,y^*)  \\
       (I-W)\mathcal{J}_y g(x^*,y^*)
         \end{array} \right]d_x=\left[ \begin{array}{c}
          \mathcal{J}y (x^*)\\
          \mathcal{J}\mu(x^*)  \\
       \mathcal{J}\xi(x^*)
         \end{array} \right]d_x.\label{4-3}
 \end{eqnarray}
Replacing $d_y$ by $\mathcal{J}y(x^*) d_x$ into the conditions $(\ref{4-2-1})$ and (\ref{mfg}), we derive that
\begin{eqnarray}
&&    \left(\mathcal{J}_xH(x^*,y^*)+\mathcal{J}_yH(x^*,y^*)\mathcal{J}y(x^*)\right)d_x=0, \label{MFCQ-b}\\
&&  \left(\nabla_xG_i(x^*,y^*)^T+\nabla_yG_i(x^*,y^*)^T\mathcal{J}y(x^*)\right)d_x<0,\quad i\in I_{G}(x^*,y^*).\label{MFCQ-c}
  \end{eqnarray}
 It shows that if the MFCQ holds at $(x^*,y^*,\mu^*,\xi^*)$ for the problem (FP), then $\mathcal{J}_xH(x^*,y^*)+\mathcal{J}_yH(x^*,y^*)\mathcal{J}y(x^*) $ has full row rank and (\ref{MFCQ-b})-(\ref{MFCQ-c}) hold,
 which  guarantee that the MFCQ holds at $x^*$ for the problem (SP).

 Conversely, if there exists $d_x\in \mathbb{R}^n$ such that (\ref{MFCQ-b})-(\ref{MFCQ-c}) hold, by setting $(d_y, d_{\mu}, d_{\xi})$ as in (\ref{4-3}), we derive that (\ref{mfg})-(\ref{4-2}) hold. Then the MFCQ holds at $(x^*,y^*,\mu^*,\xi^*)$ for (FP).
We complete the proof.
\end{proof}

The Lagrangian functions of (SP) and (FP) are defined by
\begin{eqnarray}
   &&L^{SP}(x;\lambda_H,\lambda_G)=F(x,y(x))+\lambda_H^TH(x,y(x))+\lambda_G^TG(x,y(x)),\nonumber\\
   &&L^{FP}(x,y,\mu,\xi;\lambda)=F(x,y)+\lambda_H^TH(x,y)+\lambda_G^TG(x,y)
    +\lambda_{\mathcal{L}}^T\nabla_{y} \mathcal{L}(x;y,\mu,\xi)\nonumber\\
    &&~~~~~~~~~~~~~~~~~~~~~~~~~~~~~~+\lambda_h^Th(x,y)+\lambda_g^T(g(x,y)-\Pi_{\mathbb{R}_{-}^s}(g(x,y)+\xi)),\nonumber
 \end{eqnarray}
where $\lambda=(\lambda_H,\lambda_G,\lambda_{\mathcal{L}},\lambda_h,\lambda_g)$, respectively.
The following theorem gives the first-order optimality conditions of a bi-local solution of the problem (BP), which are equivalent to the first-order optimality conditions of the problems (SP) and (FP) at a local solution.
For clarity of the notation, we define that $$ L(x,y;\lambda_H,\lambda_G)=F(x,y)+\lambda_H^TH(x,y)+\lambda_G^TG(x,y).$$
Simple calculation yields that the first-order derivative of $L^{SP}(x;\lambda_H,\lambda_G)$ at the point $(x^*;\lambda_H^*,\lambda_G^*)$ with respect to the variable $x$ :$$\nabla_{x} L^{SP}(x^*;\lambda^*_H,\lambda^*_G)= \nabla_{x}L(x^*,y(x^*);\lambda_H^*,\lambda_G^*)+\mathcal{J}y(x^*)^T\nabla_{y}L(x^*,y(x^*);\lambda_H^*,\lambda_G^*)
  .$$

\begin{theorem}[First-order Necessary Optimality Conditions]\label{thm 4-1}
Let $\left(x^*, y^*\right) \in \mathbb{R}^n \times \mathbb{R}^m$ be a point around which $f,g,h$ are continuously differentiable and twice continuously differentiable with respect to the variable $y$ and $F,G,H$ are continuously differentiable.
 Suppose that the lower level problem $(P_{x^*})$ satisfies the Jacobian uniqueness conditions at $(y^*,\mu^*,\xi^*)$ and $(x^*,y^*)$ is a bi-local solution of (BP). Then the following conditions hold:
\begin{itemize}
  \item[(i)]If the MFCQ holds at $x^*$ for (SP), then there exists $(\lambda_H^*,\lambda_G^*)\in \mathbb{R}^{p}\times\mathbb{R}^{q}$ such that
      \begin{align}
        & \nabla_{x} L^{SP}(x^*;\lambda^*_H,\lambda^*_G)=0,\notag\\
         & H(x^*,y^*)= 0,\label{4-5}\\
        &0\le \lambda_G^*\bot G(x^*,y^*)\le 0.\notag
      \end{align}
      The set of all $(\lambda_H^*,\lambda_G^*)$ satisfying (\ref{4-5}), denoted by $\Lambda(x^*)$, is a nonempty compact convex set.
  \item[(ii)]If the MFCQ holds at $(x^*,y^*,\mu^*,\xi^*)$ for (FP), then there exists $\lambda^*=(\lambda_H^*,\lambda_G^*,\lambda_{\mathcal{L}}^*,\lambda_h^*,\lambda_g^*)
  \in\mathbb{R}^{p}\times\mathbb{R}^{q}\times\mathbb{R}^{m}\
  \times\mathbb{R}^{r}\times\mathbb{R}^{s}$
   such that
      \begin{eqnarray}
        && \nabla_{(x,y,\mu,\xi)} L^{FP}(x^*,y^*,\mu^*,\xi^*;\lambda^*)=0\nonumber \\
         && H(x^*,y^*)= 0,\nonumber\\
&&0\le \lambda_G^*\bot G(x^*,y^*)\le 0,\label{4-6}\\
&&\nabla_{y} \mathcal{L}(x^*;y^*,\mu^*,\xi^*)=0,\nonumber\\
&&h(x^*,y^*)=0,\nonumber\\
&&g(x^*,y^*)-\Pi_{\mathbb{R}_{-}^s}(g(x^*,y^*)+\xi^*)=0.\nonumber
\end{eqnarray}
  The set of all $\lambda^*$ satisfying $(\ref{4-6})$, denoted by $\Lambda(x^*,y^*,\mu^*,\xi^*)$, is a nonempty compact convex set.
\end{itemize}
Furthermore, (i) and (ii) are equivalent.
\end{theorem}

\begin{proof}
Since $(x^*,y^*)$ is a bi-local solution of (BP), we can know that $x^*$ is a local solution of (SP) and $(x^*,y^*,\mu^*,\xi^*)$ is a local solution of (FP) from Theorem \ref{prop-equal BP CP}. If MFCQ holds at $x^*$ for (SP), then (\ref{4-5}) holds from \cite{Nocedal 99}. Similarly, if MFCQ holds at $(x^*,y^*,\mu^*,\xi^*)$ for (FP), then (\ref{4-6}) holds.

Next we need to prove that (i) and (ii) are equivalent. Assume there exists $(\lambda_H^*,\lambda_G^*)$ such that the system (\ref{4-5}) holds.
Since the Jacobian uniqueness conditions hold at $(y^*,\mu^*,\xi^*)$ for Problem $(P_{x^*})$,
we get
\begin{eqnarray*}
&&\nabla_{y} \mathcal{L}(x^*;y^*,\mu^*,\xi^*)=0,\nonumber\\
&&h(x^*,y^*)=0,\nonumber\\
&&g(x^*,y^*)-\Pi_{\mathbb{R}_{-}^s}(g(x^*,y^*)+\xi^*)=0.\nonumber
\end{eqnarray*}
Let
\begin{align}\label{lambda L-h-g x*}
\left[\begin{array}{c}
            \lambda_{\mathcal{L}}^* \\
          \lambda_h^*\\
           \lambda_g^*
          \end{array}\right]:= - K(x^*)^{-T}[I, 0, 0]^T
\nabla_{y}L(x^*,y(x^*);\lambda_H^*,\lambda_G^*),
          \end{align}
where $K(x)^{-T}$ denotes the transpose of the inverse matrix $K(x)^{-1}$.
It follows from Proposition \ref{prop3.1} that
  \begin{eqnarray*}
   && \mathcal{J}y(x^*)^T \nabla_{y}L(x^*,y(x^*);\lambda_H^*,\lambda_G^*)\\
&&  =-\left[
		\begin{array}{c}
			\nabla_{yx}^2 \mathcal{L}(x^*;y^*,\mu^*,\xi^*)\\
            \mathcal{J}_x h(x^*,y^*)\\
			(I-W^*)\mathcal{J}_x g(x^*,y^*)
		\end{array}\right]^T K(x^*)^{-T}[I, 0, 0]^T
\nabla_{y}L(x^*,y(x^*);\lambda_H^*,\lambda_G^*)\\
&&=\nabla_{xy}^2 \mathcal{L}(x^*,y^*;\lambda_H^*,\lambda_G^*) \lambda_{\mathcal{L}}^*
+\mathcal{J}_x h(x^*,y^*)^T\lambda_h^*  + \mathcal{J}_x g(x^*,y^*)^T(I-W^*)\lambda_g^*.
\end{eqnarray*}
From (\ref{4-5}), we have that
 \begin{align}\label{kktsp}
 0&=  \nabla_{x} L^{SP}(x^*;\lambda_H^*,\lambda_G^*)\notag\\
 &= \nabla_{x}L(x^*,y(x^*);\lambda_H^*,\lambda_G^*)
 +\mathcal{J}y(x^*)^T\nabla_{y}L(x^*,y(x^*);\lambda_H^*,\lambda_G^*)\notag\\
&=\nabla_{x}L^{FP}(x^*,y^*,\mu^*,\xi^*;\lambda^*).
  \end{align}
Furthermore, (\ref{lambda L-h-g x*}) means that
\begin{align*}
0= K(x^*)^T  \left[\begin{array}{c}
            \lambda_{\mathcal{L}}^* \\
            \lambda_h^*\\
            \lambda_g^*
          \end{array}
 \right]+
 \left[\begin{array}{c}
                                \nabla_{y}L(x^*,y^*;\lambda_H^*,\lambda_G^*) \\
                                 0 \\
                                 0
                               \end{array}\right],
\end{align*}
which implies that
\begin{align}
\label{nabla_yL CP}
& 0=\nabla_{y}L(x^*,y^*;\lambda_H^*,\lambda_G^*)+\nabla_{yy}^2 \mathcal{L}(x^*;y^*,\mu^*,\xi^*) \lambda_{\mathcal{L}}^* +\mathcal{J}_y h(x^*,y^*)^T\lambda_h^*+ \mathcal{J}_y g(x^*,y^*)^T(I-W^*)\lambda_g^*\notag\\
  &\enspace=\nabla_yL^{FP}(x^*,y^*,\mu^*,\xi^*;\lambda^*), \\
  \label{nabla_mu L CP}
   & 0= \mathcal{J}_y h(x^*,y^*)\lambda_{\mathcal{L}}^*=\nabla_{\mu}L^{FP}(x^*,y^*,\mu^*,\xi^*;\lambda^*),
   \\
   \label{nabla_xi L CP}
  &0=\mathcal{J}_y g(x^*,y^*)\lambda_{\mathcal{L}}^*-W^*\lambda_g^*=
  \nabla_{\xi}L^{FP}(x^*,y^*,\mu^*,\xi^*;\lambda^*).
\end{align}
The conditions (\ref{kktsp})-(\ref{nabla_xi L CP}) imply that
 $0=\nabla_{(x,y,\mu,\xi)}L^{FP}(x^*,y^*,\mu^*,\xi^*;\lambda^*)$
 and hence there exists $\lambda^*=(\lambda_H^*,\lambda_G^*,\lambda_{\mathcal{L}}^*,\lambda_h^*,\lambda_g^*)$ such that the system (\ref{4-6}) holds.

 Conversely if there exists $\lambda^*$ such that $(\ref{4-6})$ holds.  $\nabla_{(y,\mu,\xi)}L^{FP}(x^*,y^*,\mu^*,\xi^*;\lambda^*)=0$ implies that $(\lambda_{\mathcal{L}}^*,\lambda_h^*,\lambda_g^*)$ is defined in (\ref{lambda L-h-g x*}) since the matrix $K(x^*)$ is invertible. Combining this with (\ref{kktsp}), we can derive that $ \nabla_{x} L^{SP}(x^*;\lambda_H^*,\lambda_G^*)=0$. Thus the converse conclusion is true. The proof is completed.
\end{proof}

In the rest of this subsection, we study the second-order optimality conditions  of the bilevel problem.
 Let $u=(x,y,\mu,\xi)$ and $u^*=(x^*,y^*,\mu^*,\xi^*)$. Denote by $\nabla_{(x,y)}^2 L$ the second-order derivative of the function $L$ with respect to the variable $(x,y)$ for convince.
For clarity of the following proof, we give the first and second-order derivatives of a general composite function $\tilde{\theta}(x)=\theta(x,y(x))$ with respect to the variable $x$:
\begin{align*}
  \nabla \tilde{\theta}(x)=&\nabla_x \theta(x,y(x))+\mathcal{J}y(x)^T\nabla_{y}\theta(x,y(x))=
\left[ \begin{array}{cc}
         I_n & \mathcal{J}y(x)^T
         \end{array}\right]
\nabla_{(x,y)}\theta(x,y(x))
\end{align*}
and
\begin{align}\label{nabla xx theta}
  \nabla^2 \tilde{\theta}(x)=&\left[ \begin{array}{cc}
         I_n & \mathcal{J}y(x)^T
         \end{array}\right]
\nabla_{(x,y)}^2\theta(x,y(x))\left[ \begin{array}{c}
         I_n \\ \mathcal{J}y(x)
         \end{array}\right]+\sum_{i=1}^{m}\frac{\partial \theta}{\partial y_i}(x,y(x) )\nabla^2 y_i(x).
\end{align}
\begin{lemma}\label{lemma 4-1}
Let $\left(x^*, y^*\right) \in \mathbb{R}^n \times \mathbb{R}^m$ be a point around which $f,g,h$ are twice continuously differentiable and thrice continuously differentiable with respect to variable $y$ and $F,G,H$ are twice continuously differentiable. Suppose that the lower level problem $(P_{x^*})$ satisfies the Jacobian uniqueness conditions at $(y^*,\mu^*,\xi^*)$, then for $\lambda_H^*\in \mathbb{R}^p$ and $\lambda_G^*\in \mathbb{R}^q$, we have
\begin{align}
\label{nabla xx L SP x*}
 \nabla_{xx}^2 L^{SP}(x^*;\lambda_H^*,\lambda_G^*)
 = &\left[ \begin{array}{cc}
         I_n &
         \mathcal{J}y(x^*)^T
         \end{array}\right]\nabla_{(x,y)}^2 L(x^*,y^*;\lambda_H^*,\lambda_G^*)
     \left[ \begin{array}{c}
         I_n \\
         \mathcal{J}y(x^*)
         \end{array}\right]\notag\\
        & +\sum_{i=1}^{m}\frac{\partial L}{\partial y_i}(x^*,y^*;\lambda_H^*,\lambda_G^*)\nabla^2 y_i(x^*)\\
        \label{nabla xx L SP-FP x*}
        = &U(x^*)^T\nabla_{uu}^2L^{FP}(x^*,y^*,\mu^*,\xi^*;\lambda^*)U(x^*),
         \end{align}
 where $U(x):=\left[\begin{array}{cccc}
         I_n &
         \mathcal{J}y(x)^T&
         \mathcal{J}\mu(x)^T&
          \mathcal{J}\xi(x)^T
                   \end{array}\right]^T$,
 $\mathcal{J}y(x), \mathcal{J}\mu(x), \mathcal{J}\xi(x)$ are given by (\ref{nabla-y}), $\lambda^*=(\lambda_H^*,\lambda_G^*,\lambda_{\mathcal{L}}^*,\lambda_h^*,\lambda_g^*)$ and $(\lambda_{\mathcal{L}}^*,\lambda_h^*,\lambda_g^*)$ is defined by (\ref{lambda L-h-g x*}).
\end{lemma}
\begin{proof}
Recall that $L^{SP}(x;\lambda_H,\lambda_G)=L(x,y(x);\lambda_H,\lambda_G)$, then it follows from  (\ref{nabla xx theta}) that
(\ref{nabla xx L SP x*}) holds.
Noticing that $L(x,y;\lambda_H,\lambda_G)$ is not related with the variables $\mu$ and $\xi$, then $\nabla_{xx}^2L^{SP}(x^*;\lambda^*_H,\lambda_G^*)$ can also be expressed as
\begin{align}\label{nabla xx L SP x}
  &\nabla_{xx}^2L^{SP}(x^*;\lambda^*_H,\lambda_G^*)\notag\\
=&
U(x^*)^T
\nabla_{(x,y,\mu,\xi)}^2L(x^*,y^*;\lambda_H^*,\lambda_G^*)U(x^*)
+\sum_{i=1}^{m}\frac{\partial L}{\partial y_i}(x^*,y^*;\lambda_H^*,\lambda_G^*)\nabla^2 y_i(x^*)\notag\\
         &+\sum_{i=1}^{r}\frac{\partial L}{\partial \mu_i}(x^*,y^*;\lambda_H^*,\lambda_G^*)\nabla^2 \mu_i(x^*)
         +\sum_{i=1}^{s}\frac{\partial L}{\partial \xi_i}(x^*,y^*;\lambda_H^*,\lambda_G^*)\nabla^2 \xi_i(x^*).
\end{align}

From Lemma \ref{lemma-IFT} and Remark \ref{second IFT}, $(y(x),\mu(x),\xi(x))$ is a KKT point of Problem $(P_x)$ and twice continuously differentiable, then
\begin{gather}
\label{y j}
   \frac{\partial\mathcal{L}}{\partial y_j}(x;y(x),\mu(x),\xi(x))=0,\quad j=1,\cdots,m;\\
   \label{mu k}
   h_k(x,y(x))=0,\quad k=1,\cdots, r;\\
    \label{xi l}
  \hat{g}_l(x,y(x),\mu(x),\xi(x)):= g_l(x,y(x))-\Pi_{\mathbb{R}_{-}}(g_l(x,y(x))+\xi_l(x))=0,\quad l=1,\cdots, s.
\end{gather}

Similarly to (\ref{nabla xx theta}), for every $j=1,\cdots,m$, we derive  the second derivative on both sides of (\ref{y j}) at  $x^*$ with respect to the variable $x$:
  \begin{align}\label{equ-SD-Dy-L}
 0=&U(x^*)^T
\nabla_{(x,y,\mu,\xi)}^2\left(\frac{\partial\mathcal{L}}{\partial y_j}\right)(x^*;y^*,\mu^*,\xi^*)U(x^*)+\sum_{i=1}^{m}\frac{\partial^2\mathcal{L}}{\partial y_i \partial y_j}(x^*;y^*,\mu^*,\xi^*)\nabla^2 y_i(x^*)\notag\\
         &
         +\sum_{i=1}^{r}\frac{\partial^2\mathcal{L}}{\partial \mu_i \partial y_j}(x^*;y^*,\mu^*,\xi^*)\nabla^2 \mu_i(x^*)+\sum_{i=1}^{s}\frac{\partial^2\mathcal{L}}{\partial \xi_i \partial y_j}(x^*;y^*,\mu^*,\xi^*)\nabla^2 \xi_i(x^*).
\end{align}
For every $k=1,\cdots,r$, the second derivative on both sides of (\ref{mu k}) at  $x^*$ with respect to the variable $x$ is
  \begin{align}
 \label{equ-SD-h}
       0=&U(x^*)^T
\nabla_{(x,y,\mu,\xi)}^2h_k(x^*,y^*)U(x^*)+\sum_{i=1}^{m}\frac{\partial h_k}{\partial y_i}(x^*,y^*)\nabla^2 y_i(x^*)\notag\\
         &
         +\sum_{i=1}^{r}\frac{\partial h_k}{\partial \mu_i}(x^*,y^*)\nabla^2 \mu_i(x^*)+\sum_{i=1}^{s}\frac{\partial h_k}{\partial \xi_i}(x^*,y^*)\nabla^2 \xi_i(x^*).
  \end{align}
 For every $l=1,\cdots,s$, the second derivative on both sides of (\ref{xi l}) at  $x^*$ with respect to the variable $x$ is as follows:
  \begin{align}
          \label{equ-SD-g}
           0=&U(x^*)^T
\nabla_{(x,y,\mu,\xi)}^2\hat{g}_l(x^*,y^*,\mu^*,\xi^*)U(x^*)
+\sum_{i=1}^{m}\frac{\partial \hat{g}_l}{\partial y_i}(x^*,y^*,\mu^*,\xi^*)\nabla^2 y_i(x^*)\notag\\
         &
         +\sum_{i=1}^{r}\frac{\partial \hat{g}_l}{\partial \mu_i}(x^*,y^*,\mu^*,\xi^*)\nabla^2 \mu_i(x^*)+\sum_{i=1}^{s}\frac{\partial  \hat{g}_l }{\partial \xi_i}(x^*,y^*,\mu^*,\xi^*)\nabla^2 \xi_i(x^*).
  \end{align}

Multiplying (\ref{equ-SD-Dy-L}) by $\lambda_{\mathcal{L},j}^*$ for $j=1,\cdots,m$,  multiplying (\ref{equ-SD-h}) by $\lambda_{h,k}^*$ for $k=1,\cdots,r$ and multiplying (\ref{equ-SD-g}) by $\lambda_{g,l}^*$ for $l=1,\cdots,s$, then adding them to (\ref{nabla xx L SP x}), we get
\begin{align}\label{Lxx SP}
&\nabla_{xx}^2L^{SP}(x^*;\lambda^*_H,\lambda_G^*)\notag\\
 =&U(x^*)^T\nabla_{(x,y,\mu,\xi)}^2L^{FP}(x^*,y^*,\mu^*,\xi^*;\lambda^*)U(x^*)
+\sum_{i=1}^{m}\frac{\partial L^{FP}}{\partial y_i}(x^*,y^*,\mu^*,\xi^*;\lambda^*)\nabla^2 y_i(x^*)\notag\\
& +\sum_{i=1}^{r}\frac{\partial L^{FP}}{\partial \mu_i}(x^*,y^*,\mu^*,\xi^*;\lambda^*)\nabla^2 \mu_i(x^*)
+\sum_{i=1}^{s}\frac{\partial L^{FP}}{\partial \xi_i}(x^*,y^*,\mu^*,\xi^*;\lambda^*)\nabla^2 \xi_i(x^*),
\end{align}
 where
 \begin{align*}
    L^{FP}(x,y,\mu,\xi;\lambda)=&L(x,y;\lambda_H,\lambda_G)
    +\sum_{j=1}^{m}\lambda_{\mathcal{L},j}\frac{\partial\mathcal{L}}{\partial y_j}(x;y,\mu,\xi)
    +\sum_{k=1}^{r}\lambda_{h,k}h_k(x,y)+\sum_{l=1}^{s}\lambda_{g,l}\hat{g}_l(x,y,\mu,\xi).
 \end{align*}


Since $(\lambda_{\mathcal{L}}^*,\lambda_h^*,\lambda_g^*)$ is defined by (\ref{lambda L-h-g x*}),
it follows from the proof of Theorem \ref{thm 4-1} that
\begin{align*}
  \nabla_{(y,\mu,\xi)}L^{FP}(x^*,y^*,\mu^*,\xi^*;\lambda^*)=0.
\end{align*}
Then substituting this into
 (\ref{Lxx SP}), we can derive that (\ref{nabla xx L SP-FP x*}) holds.
   The proof is complete.
\end{proof}

Noticing that the second derivative of $L^{FP}(u,\lambda^*)$ at point $u^*$ with respect to $u$ is specificly represented as
\begin{eqnarray*}\label{array gamma}
\nabla_{uu}^2L^{FP}(x^*,y^*,\mu^*,\xi^*;\lambda^*)=
 \left[\begin{array}{cc}
           \Gamma_{11} &  \Gamma_{12} \\
            \Gamma_{12}^T &  \Gamma_{22}
         \end{array}\right],
 \end{eqnarray*}
where
  \begin{align*}
    &\Gamma_{11} =\nabla_{(x,y)}^2L^{FP}(x^*,y^*,\mu^*,\xi^*;\lambda^*); \\
   &\Gamma_{12}=\left[\begin{array}{cccccc}
                                  \nabla_{xy}^2h_1(x^*,y^*)\lambda_{\mathcal{L}}^*  & \cdots & \nabla_{xy}^2h_r(x^*,y^*)\lambda_{\mathcal{L}}^*
                                  &\nabla_{xy}^2g_1(x^*,y^*)\lambda_{\mathcal{L}}^*  & \cdots & \nabla_{xy}^2g_s(x^*,y^*)\lambda_{\mathcal{L}}^*
                                  \\
                                  \nabla_{yy}^2h_1(x^*,y^*)\lambda_{\mathcal{L}}^*  & \cdots & \nabla_{yy}^2h_r(x^*,y^*)\lambda_{\mathcal{L}}^*
                                  &\nabla_{yy}^2g_1(x^*,y^*)\lambda_{\mathcal{L}}^*  & \cdots & \nabla_{yy}^2g_s(x^*,y^*)\lambda_{\mathcal{L}}^*
                                 \end{array}\right];\notag\\
  &\Gamma_{22}=0,\notag
 \end{align*}
with
\begin{align}\label{}
 L^{FP}(x,y,\mu,\xi;\lambda)=&F(x,y)+\lambda_H^TH(x,y)+\lambda_G^TG(x,y)+\lambda_{\mathcal{L}}^T\left(\nabla_{y}f(x,y)+\mathcal{J}_y h(x,y)^T\mu+\mathcal{J}_y g(x,y)^T\xi\right)\notag\\
 &+\lambda_h^Th(x,y)+\lambda_g^T\left(g(x,y)-\Pi_{\mathbb{R}_{-}^s}(g(x,y)+\xi)\right).\notag
\end{align}
Recall that $\nabla_{(x,y)}^2 L$ denotes the second-order derivative of the function $L$ with respect to the variable $(x,y)$.

We are now ready to study the second-order optimality conditions for the bilevel problems.
The critical cone of (SP) at $x^*$ is defined by
\begin{align}\label{critical cone-SP}
 \mathcal{C}(x^*) =   \{d_x\in \mathbb{R}^n:&\left(\mathcal{J}_xH(x^*,y^*)+\mathcal{J}_yH(x^*,y^*)\mathcal{J}y(x^*)\right)d_x=0;\notag \\
   & \left(\nabla_xG_i(x^*,y^*)^T+\nabla_yG_i(x^*,y^*)^T\mathcal{J}y(x^*)\right)d_x\le0,  i\in I_{G}(x^*,y^*);\notag\\
   &\left(\nabla_xF(x^*,y^*)^T+\nabla_yF(x^*,y^*)^T\mathcal{J}y(x^*)\right)d_x\le0\}
\end{align}
and the critical cone of (FP) at $(x^*,y^*,\mu^*,\xi^*)$ is defined  by
\begin{align}\label{critical cone-CP}
\mathcal{C}(x^*,y^*,\mu^*,\xi^*)=\{d\in \mathbb{R}^{n+m+r+s}:
   & \nabla_xG_i(x^*,y^*)^Td_x+\nabla_yG_i(x^*,y^*)^Td_y\le0,  i\in I_{G}(x^*,y^*);\notag\\
   &Ad=0;\quad
\nabla_xF(x^*,y^*)^Td_x+\nabla_yF(x^*,y^*)^Td_y\le0  \},
\end{align}
where $A$ is defined by (\ref{A}).

\begin{theorem}[Second-order Necessary Optimality Conditions]\label{thm 4-2}
Let $\left(x^*, y^*\right) \in \mathbb{R}^n \times \mathbb{R}^m$ be a point around which $f,g,h$ are twice continuously differentiable and thrice continuously differentiable with respect to variable $y$ and $F,G,H$ are twice continuously differentiable.
 Suppose that the lower level problem $(P_{x^*})$ satisfies the Jacobian uniqueness conditions at $(y^*,\mu^*,\xi^*)$ and $(x^*,y^*)$ is a bi-local solution of (BP). Then the following conditions hold:
\begin{itemize}
  \item[(i)]If the MFCQ holds at $x^*$ for (SP), then $\Lambda(x^*)$ is a nonempty compact convex set and for every $d_x\in \mathcal{C}(x^*)$, 
\begin{equation*}
 \max_{(\lambda_H^*,\lambda_G^*)\in \Lambda(x^*)}\left\{d_x^T \nabla_{xx}^2L^{SP}(x^*,\lambda_H^*,\lambda_G^*)d_x\right\}\ge0.
 \end{equation*}
  \item[(ii)]If the MFCQ holds at $(x^*,y^*,\mu^*,\xi^*)$ for (FP), then $\Lambda(x^*,y^*,\mu^*,\xi^*)$ is a nonempty compact convex set and for every $d\in \mathcal{C}(x^*,y^*,\mu^*,\xi^*)$,
  \begin{equation*}
 \max_{\lambda^*\in \Lambda(x^*,y^*,\mu^*,\xi^*)}\left\{d^T \nabla_{uu}^2L^{FP}(u^*,\lambda^*)d\right\}\ge0.
 \end{equation*}
\end{itemize}
Furthermore, (i) and (ii) are equivalent.
\end{theorem}

\begin{proof}
Since $(x^*,y^*)$ is a bi-local solution of (BP), we know that $x^*$ is a local solution of (SP) and $(x^*,y^*,\mu^*,\xi^*)$ is a local solution of (FP) from Theorem \ref{prop-equal BP CP}, we obtain the second-order necessary optimality conditions of (SP) and (FP) from \cite{Nocedal 99}, then (i) and (ii) hold.

We now show that (i) and (ii) are equivalent. Assume that for any $d\in \mathcal{C}(x^*,y^*,\mu^*,\xi^*)$,
it follows from the proof of the Proposition \ref{prop 4-1} that $Ad=0$, i.e.,
\begin{gather}\label{4-7-1}
\mathcal{J}_xH(x^*,y^*)d_x+\mathcal{J}_yH(x^*,y^*)d_y=0, \\
\quad d_y=\mathcal{J}y(x^*)d_x,\quad d_{\mu}=\mathcal{J}\mu(x^*)d_x,\quad d_{\xi}=\mathcal{J}\xi(x^*)d_x\notag,
\end{gather}
then substituting $d_y=\mathcal{J}y(x^*)d_x$
into $(\ref{4-7-1})$ and $(\ref{critical cone-CP})$, we can derive that $ d_x\in \mathcal{C}(x^*) $. Conversely for any $d_x\in \mathcal{C}(x^*)$, let $d:=U(x^*)d_x$, then $d\in \mathcal{C}(x^*,y^*,\mu^*,\xi^*) $ from (\ref{critical cone-SP}). 

 It follows from Theorem \ref{thm 4-1} that $(\lambda_H^*,\lambda_G^*)\in \Lambda(x^*)$ if and only if
 $\lambda^*=(\lambda_H^*,\lambda_G^*,\lambda_{\mathcal{L}}^*,\lambda_h^*,\lambda_g^*)\in\Lambda(x^*,y^*,\mu^*,\xi^*)$, where
 $(\lambda_{\mathcal{L}}^*,\lambda_h^*,\lambda_g^*)$ is defined by (\ref{lambda L-h-g x*}). Thus from Lemma \ref{lemma 4-1},
\begin{align}
\max_{(\lambda_H^*,\lambda_G^*)\in \Lambda(x^*)}&\left\{d_x^T \nabla_{xx}^2L^{SP}(x^*,\lambda_H^*,\lambda_G^*)d_x\right\}\nonumber\\
=\max_{\lambda^*\in\Lambda(x^*,y^*,\mu^*,\xi^*)}&\left\{d_x^T U(x^*)^T
 \nabla_{uu}^2L^{FP}(u^*;\lambda^*)
    U(x^*)d_x\right\}\nonumber\\
 =\max_{\lambda^*\in\Lambda(x^*,y^*,\mu^*,\xi^*)}&\left\{ d^T
 \nabla_{uu}^2L^{FP}(u^*;\lambda^*)
     d\right\}\label{4-8}.
\end{align}
 Therefore (i) and (ii) are equivalent. We complete the proof.
\end{proof}

\begin{theorem}[Second-order Sufficient Optimality Conditions]\label{ssosc}
 Let $\left(x^*, y^*\right) \in \mathbb{R}^n \times \mathbb{R}^m$ be a point around which $f,g,h$ are twice continuously differentiable and thrice continuously differentiable with respect to variable $y$ and $F,G,H$ are twice continuously differentiable.
Assume that $(x^*,y^*)\in \Phi$ and $y^* \in Y(x^*)$. Let $(\mu^*,\xi^*) \in \mathbb{R}^r\times\mathbb{R}^s$.
Suppose that Problem $(P_{x^*})$ satisfies the Jacobian uniqueness conditions at $\left(y^*, \mu^*,\xi^*\right)$. Then the following statements are equivalent:
  \begin{itemize}
    \item[(i)] $\Lambda(x^*)\neq \emptyset$, and for every $d_x\in \mathcal{C}(x^*)\backslash \{0\}$,
\begin{equation}
 \sup_{(\lambda_H^*,\lambda_G^*)\in \Lambda(x^*)}\left\{d_x^T \nabla_{xx}^2L^{SP}(x^*,\lambda_H^*,\lambda_G^*)d_x\right\}>0.\notag
 \end{equation}
    \item [(ii)] $\Lambda(x^*,y^*,\mu^*,\xi^*)\neq \emptyset$, and for every $d\in \mathcal{C}(x^*,y^*,\mu^*,\xi^*)\backslash \{0\}$,
\begin{equation}\label{SOC-CP}
 \sup_{\lambda^*\in \Lambda(x^*,y^*,\mu^*,\xi^*)}\left\{d^T \nabla_{uu}^2L^{FP}(u^*,\lambda^*)d \right\}>0.
 \end{equation}
  \end{itemize}
  Furthermore, if (i) or (ii) holds, there exist $\delta_2 \in(0, \delta_1), \varepsilon_2 \in(0, \varepsilon_1)$ (where $\delta_1$ and $\varepsilon_1$ are given by Lemma \ref{lemma-IFT}) and $\gamma>0$ such that for $x \in \boldsymbol{B}_{\delta_2}(x^*)$, $y(x) =argmin_y\{f(x,y):y\in Y(x)\cap \boldsymbol{B}_{\varepsilon_2}(y^*) \} $ and $(x,y(x))\in \Phi$,
\begin{equation}\label{equ-growth}
  F(x,y(x))\ge F(x^*,y^*)+\gamma \|x-x^*\|^2,
\end{equation}
which indicates that $\left(x^*, y^*\right)$ is a bi-local minimum point of (BP).
\end{theorem}
\begin{proof}
Since the Jacobian uniqueness conditions hold at $(y^*,\mu^*,\xi^*)$ for Problem $(P_{x^*})$, we can know that
$(\lambda_H^*,\lambda_G^*)\in \Lambda(x^*)$ if and only if
$\lambda^*\in\Lambda(x^*,y^*,\mu^*,\xi^*)$, $d_x\in \mathcal{C}(x^*)$ if and only if $d=U(x^*)d_x\in\mathcal{C}(x^*,y^*,\mu^*,\xi^*)$ from the proof of Theorem \ref{thm 4-2}. Then we obtain (i) and (ii) are equivalent by (\ref{4-8}).

Suppose that (i) holds. From Lemma \ref{lemma-IFT}, we know that $y(x)$ is the unique local solution of Problem ($P_{x}$) for $x\in \boldsymbol{B}_{\delta_1}(x^*)$. It follows from \cite{Bonnans} that the second-order sufficient optimality conditions at $x^*$ guarantee the local quadratic growth condition. Then there exist $\delta_2\in (0,\delta_1)$, $\varepsilon_2 \in(0, \varepsilon_1)$ and $\gamma>0$ such that for any $x\in \boldsymbol{B}_{\delta_2}(x^*)$ and $(x,y(x))\in \Phi$,
  \begin{equation*}
     F(x,y(x))\ge F(x^*,y^*)+\gamma\| x-x^*\|^2,
  \end{equation*}
  where $ y(x)={\rm argmin}_y\{f(x,y):y\in Y(x) \cap \boldsymbol{B}_{\varepsilon_2}(y^*)\}$.
Thus we can obtain (\ref{equ-growth}). The proof is complete.
\end{proof}

 The item $\nabla^2 y(x)$ in the second-order derivative of $L^{SP}$ in Lemma \ref{lemma 4-1} can not be calculated due to the implicit structure of $y(x)$. Consequently, the second-order optimality conditions in the Theorem \ref{ssosc} (i) are hardly to check.
 However, it is worth noticing that
the second-order derivative of $L^{FP}$ only contains the second-order derivatives of the defining functions in (BP) and the corresponding Lagrange multipliers.
As a result, the second-order optimality conditions in the Theorem \ref{ssosc} (ii)  can be used to obtain bi-local solutions of (BP) and guarantee the convergence rates of many numerical algorithms, such as the classical augmented Lagrangian algorithm, which is stated in the rest of this section.
\subsection{The convergence rate of the augmented Lagrangian method}\label{S 4.3}
The augmented Lagrangian method is known as a method of multipliers for solving nonlinear optimization problems (NLP) with constraints. It was proposed by Hestenes \cite{Hestenes 69} and Powell \cite{Powell 69} for solving equality constrained optimization problems and was generalized by Rockafellar \cite{Rockafellar 73} to optimization problems with both equality and inequality constraints. In \cite[Chapter 3]{Bertsekas 82}, Bertsekas established local convergence of the augmented Lagrangian method for NLP if the LICQ, strict complementarity condition and SOSC are satisfied. The convergence rate is Q-linear when the penalty parameter is finite and Q-superlinear when penalty parameter  converges to infinity.
Without strict complementarity, a stronger version of second-order sufficiency is employed to establish the linear convergence rate, see \cite{Conn,Contesse-Becker,Ito}. A remarkable progress was achieved by Fern$\acute{\text{a}}$ndez and Solodov in \cite{Fernandez 12} establishing the local convergence rate of the classical augmented Lagrangian algorithm under the SOSC only.

In this subsection, we apply the augmented Lagrangian method to the problem (FP), under the SOSC of (FP), which is defined in (\ref{SOC-CP}) of Theorem \ref{ssosc} (ii), if the dual starting point is close to a multiplier satisfying SOSC, then the iteration points converge to a bi-local solution of (BP) with the Q-linear convergence rate when the penalty parameter  is finite and Q-superlinear convergence when penalty parameter  converges to infinity.

 The augmented Lagrangian function for (FP) is defined by $L_{\rho}^{FP}:\mathbb{R}^n\times\mathbb{R}^m\times\mathbb{R}^r\times\mathbb{R}^s\times\mathbb{R}^{p+q+m+r+s}\rightarrow \mathbb{R}$,
 \begin{align}\label{augmented L}
  L_{\rho}^{FP}&(x,y,\mu,\xi;\lambda)=F(x,y)+\lambda_{H}^TH(x,y)+\frac{\rho}{2}\|H(x,y)\|^2+\frac{1}{2\rho}
  \left(\|\max\{0,\lambda_{G}+\rho G(x,y)\}\|^2\right.
\notag\\
  &\left.-\|\lambda_G\|^2\right)
  +\lambda_{\mathcal{L}}^T\nabla_y \mathcal{L}(x;y,\mu,\xi)+\frac{\rho}{2}\|\nabla_y \mathcal{L}(x;y,\mu,\xi)\|^2+\lambda_{h}^Th(x,y)+\frac{\rho}{2}\|h(x,y)\|^2
 \notag\\
 & +\lambda_{g}^T(g(x,y)-\Pi_{\mathbb{R}_{-}^s}(g(x,y)+\xi))+\frac{\rho}{2}\|g(x,y)
  -\Pi_{\mathbb{R}_{-}^s}(g(x,y)+\xi)\|^2 .\notag
 \end{align}
 To simplify the notation, we define $u^k=(x^k,y^k,\mu^k,\xi^k)$, $K=\{0\}^p \times\mathbb{R}^q_{-}\times\{0\}^m \times\{0\}^r \times\{0\}^s$ and the
 constraint region of (FP) can be formulated as:
\begin{eqnarray*}
\widetilde{G}(u):=\left(
\begin{array}{c}
  H(x,y) \\
  G(x,y)\\
 \nabla_y \mathcal{L}(x;y,\mu,\xi) \\
  h(x,y) \\
  g(x,y)-\Pi_{\mathbb{R}_{-}^s}(g(x,y)+\xi)
\end{array}
\right)\in K.
\end{eqnarray*}
Then the violation of the KKT conditions (\ref{4-6}) is measured by the natural residual $\sigma$ : $\mathbb{R}^{n+m+r+s} \times \mathbb{R}^{p+q+m+r+s} \rightarrow \mathbb{R}_{+}$,
$$
\sigma(u,\lambda)=\left\|\left[\begin{array}{c}
\nabla_u L^{FP}(u;\lambda) \\
\lambda-\Pi_{K^{\circ}}(\lambda+\widetilde{G}(u))
\end{array}\right]\right\| ,
$$
where $K^{\circ}=\mathbb{R}^{p} \times\mathbb{R}^q_{+}\times\mathbb{R}^m\times\mathbb{R}^r\times\mathbb{R}^s$ is the polar cone of $K$.

 For an arbitrary constant $\hat{c} \in(0,+\infty)$, we consider the following iterative procedure. Given a current iterate $(u^k,\lambda^k) \in \mathbb{R}^{n+m+r+s}\times K^{\circ}$ with $\sigma\left(u^k,\lambda^k\right)>0$, a penalty parameter $\rho_k>0$, and an approximation parameter $\epsilon_k \geq 0$,
 \begin{eqnarray}
&&\text{find}\enspace u^{k+1}\enspace \text{satisfying}\enspace \left\|\nabla_{u}L_{\rho}^{FP}(u^{k+1};\lambda^k) \right\| \leq \epsilon_k,\label{5-1}\\
&&\left\|\left[\begin{array}{c}
u^{k+1}-u^k \\
\Pi_{K^{\circ}}\left(\lambda^k+\rho_k \widetilde{G}(u^{k+1})\right)-\lambda^k
\end{array}\right]\right\| \leq \hat{c} \sigma(u^k, \lambda^k),\\
&&\text { and set } \enskip\lambda^{k+1}=\Pi_{K^{\circ}}\left(\lambda^k+\rho_k \widetilde{G}(u^{k+1})\right) \text {. }\label{5-2}
 \end{eqnarray}

The following theorem reveals that the iteration sequence which generates from the procedures (\ref{5-1})-(\ref{5-2}) converges to a bi-local solution of (BP). The convergence rate under the  SOSC  directly from \cite[Theorem 3.4]{Fernandez 12}. Note that our SOSC (\ref{SOC-CP}) can be verifiable.

\begin{theorem}
  Let $(u^*,\lambda^*)$ satisfy SOSC (\ref{SOC-CP}), and let $\psi: \mathbb{R}_+\rightarrow \mathbb{R}_+ $ be any function $\psi(t) = o(t)$.
Then there exist $\bar{\epsilon},\bar{\rho}>0$ such that if $(u^0,\lambda^0)\in \boldmath{B}_{\bar{\epsilon}}(u^*,\lambda^*)\cap (\mathbb{R}^{n+m+r+s}\times K^{\circ})$ and the sequence $(u^k,\lambda^k)$ is generated according to (\ref{5-1})-(\ref{5-2}) with $\hat{c} $ large enough, $\rho_k\ge \bar{\rho}$ for all k
and $\epsilon_k=\psi(\sigma(u^k, \lambda^k))$, the following assertions hold:
\begin{itemize}
  \item[(i)]The sequence $\{(u^k,\lambda^k)\}$ is well-defined and converges to $(u^*,\bar{\lambda})$, with some $\bar{\lambda}\in \Lambda(u^*)$.
  \item[(ii)] For any $q\in(0, 1)$ there exists $\bar{\rho}_q$ such that if $\rho_k\ge\bar{\rho}_q$ for all k, then the convergence rate of $\{(u^k,\lambda^k)\}$ to $(u^*,\bar{\lambda})$ is Q-linear with quotient q.
  \item[(iii)]If $\rho_k \rightarrow +\infty$, the convergence rate is Q-superlinear.
\end{itemize}
\end{theorem}

\section{Conclusions}
In this paper, we introduce the concept of bi-local solution for bilevel programs. This solution is a variant of the traditional local solution, and it is particularly useful when the lower level problem is not convex. By analyzing the Jacobian uniqueness conditions for the lower level problem, we establish that the local solutions of (SP) and (FP) can be regarded as bi-local solutions for (BP). Additionally, we provide optimality conditions for the bi-local solutions of bilevel programs, including first-order necessary conditions and second-order necessary and sufficient conditions. These second-order conditions are easy to verify. Finally, we demonstrate the application of our theoretical results to the classical augmented Lagrangian algorithm and derive a primal-dual Q-linear convergence rate.


%




\end{document}